\documentclass[pdflatex,sn-mathphys-num]{sn-jnl}% Math and Physical Sciences Numbered Reference Style 
\usepackage{braket}
\usepackage{enumerate}
\usepackage{graphicx}%
\usepackage{multirow}%
\usepackage{amsmath,amssymb,amsfonts}%
\usepackage{amsthm}%
\usepackage{mathrsfs}%
\usepackage[title]{appendix}%
\usepackage{xcolor}%
\usepackage{textcomp}%
\usepackage{manyfoot}%
\usepackage{booktabs}%
\usepackage{algorithm}%
\usepackage{algorithmicx}%
\usepackage{algpseudocode}%
\usepackage{listings}%
\usepackage{tabularray}
\usepackage{float}
\theoremstyle{thmstyleone}%
\newtheorem{theorem}{Theorem}
\newtheorem{proposition}[theorem]{Proposition}
\newtheorem{corollary}{Corollary}
\newtheorem{lemma}{Lemma}

\theoremstyle{thmstyletwo}%
\newtheorem{remark}{Remark}%

\theoremstyle{thmstylethree}%

\begin{document}

\title[Article Title]{$L^p$ boundedness, r-nuclearity and approximation of pseudo-differential operators on $\hbar\mathbb{Z}^n$}


\author*[1]{\fnm{Juan Pablo} \sur{Lopez}}\email{juan.lopez.holguin@correounivalle.edu.co}

\affil*[1]{\orgdiv{Departamento de matematicas}, \orgname{Universidad del Valle}, \orgaddress{ \city{Cali}, \postcode{760042}, \state{Valle del Cauca}, \country{Colombia}}}

\abstract{In this work sufficient conditions on the order of the symbol are developed to ensure boundedness, compactness and r-nuclearity of pseudo-differential operators in $\hbar\mathbb{Z}^n$. In addition, these conditions allow us to obtain growth estimates for the eigenvalues of some elliptic operators, in particular perturbed discrete Schrödinger operator.}


\maketitle

\section{Introduction}\label{sec1}
In this paper we show some spectral properties for Schrödinger operators 
\begin{equation}\label{Schrodinger op}
(H_{\hbar,V}f)(k)=\Delta f+ V(k)f(k) \text{  for } k\in\hbar\mathbb{Z}^n \text{ and } V(\cdot)\geq0,
\end{equation}
where $\Delta$ is the discrete Laplacian defined as 
\[ (\Delta f)(k)=\frac{1}{\hbar^2}\sum_{j=1}^{n}f(k+\hbar e_j)-2f(k)+f(k-\hbar e_j),   \]
and $V(k)\geq 0$ being polynomial of order $\mu>0$. A special case are the discrete anharmonic oscillators, where $V(k)=|k|^{2l}$ and $l$ being a natural number.\\
\indent There has been intensive research on discrete Schrödinger operators \cite{rabinovich2008essential,rabinovich2004pseudodifference,rabinovich2010exponential,barrera2019numerical,albeverio2004schrodinger}, but the author do not find any reference with polynomial potentials. The importance of the discrete case lies in the computations based in finite differences, if we use these methods to compute the eigenvalues of Schrodinger operator on $\mathbb{R}^n$, we will obtain an approximation of the eigenvalues of \eqref{Schrodinger op}.\\ 

\indent This work is inspired by the spectral study of anharmonic operators presented in \cite{chatzakou2021class}, then sufficient conditions on the order of the symbol are developed to ensure boundedness, compactness and r-nuclearity of pseudo-differential operators in $\hbar\mathbb{Z}^n$. Additionally, we use these conditions to prove the r-nuclearity of the operator $(H_{\hbar,V}+\lambda I)^{-1}$ and conclude growth estimates for the eigenvalues of $H_{\hbar,V}$.\\
\indent Our work is based on the discrete pseudo-differential theory presented in \cite{molahajloo2010pseudo,botchway2020difference,botchway2024semi}.  Given a measurable function $\sigma$ on $\hbar\mathbb{Z}^n\times\mathbb{T}^n$, we define the pseudo-differential operator $T_\sigma$ acting on functions $a:\hbar\mathbb{Z}^n\to \mathbb{C}$ with as
\[(T_\sigma a )(k)=\int_{\mathbb{T}^n} e^{2\pi\frac{i}{\hbar}k\theta}\sigma (k,\theta) \hat{a}(\theta)d\theta, \]
where $\hat{a}$ is the Fourier transform $\mathcal{F}_{\hbar\mathbb{Z}^n}a$ defined as
\[(\mathcal{F}_{\hbar\mathbb{Z}^n}a)(\theta)=\sum_{k\in\hbar\mathbb{Z}^n} a(k)e^{-2\pi\frac{i}{\hbar}k\theta}. \]
\noindent These operators can be interpreted as the discretization of the input/output functions of pseudo-differential operators. Furthermore, the finite difference operators are included in this class \cite{botchway2020difference}.\\
\indent Additionally, we can associate an infinite matrix to 
$T_\sigma$ as in \cite{botchway2024semi}, using the representation
\begin{equation}\label{matrix operator}
(T_\sigma a )(k)=\sum_{m\in\hbar\mathbb{Z}^n} (\mathcal{F}_{\mathbb{T}^n} \sigma)(k,k-m)a(m),
\end{equation}
where the matrix is $(\sigma_{k,m})=(\mathcal{F}_{\mathbb{T}} \sigma)(k,k-m)$ and the Fourier transform in $\mathbb{T}$ is defined as
\begin{equation}\label{Fourier sigma}
(\mathcal{F}_{\mathbb{T}^n}\sigma)(k,m)=\int_{\mathbb{T}^n} \sigma(k,\theta)e^{-2\pi\frac{i}{\hbar}m\cdot \theta}d\theta.
\end{equation}
\noindent On the other hand, there exists extensive research in eigenvalues and invertibility of infinite matrices \cite{shivakumar1987eigenvalues,shivakumar2009review,koehler1957estimates,boffi2010finite}, for this reason, we are motivated to establish properties in terms of the associated matrix of a pseudo-differential operator.\\
\indent The main objective of the present paper is to establish criteria for boundedness and $r-nuclearity$ properties in terms of the associated matrix of the pseudo-differential operator. The following proposition is a compilation of the results in \cite{molahajloo2010pseudo,delgado2013p} and will help us elucidate our main theorem.
\begin{proposition}\label{Base theorem}
    Let $\sigma$ be a measurable function such that $|(\mathcal{F}_{\mathbb{T}} \sigma)(k,m)|\leq C(k)\omega(m)$. The following conditions are sufficient to guarantee that operator $T_\sigma:L^p(\mathbb{Z})\to L^p(\mathbb{Z}) $ with $1\leq p<\infty$ belong to different ideals of operators as correspondingly specified: \begin{enumerate}[(i).]
        \item If $C(\cdot)\in L^\infty(\mathbb{Z})$ and $\omega\in L^1(\mathbb{Z}) $, then $T_\sigma$ is bounded \cite[Theorem 3.2]{molahajloo2010pseudo}.\label{Base theorem i}
        \item If $C(n)\to 0 $ with $|n|\to\infty$ and $\omega\in L^1(\mathbb{Z}) $, then $T_\sigma$ is compact \cite[Theorem 3.3]{molahajloo2010pseudo}.\label{Base theorem ii}
        \item If $C(\cdot)\in L^1(\mathbb{Z})$ and $\omega\in L^p(\mathbb{Z}) $, then $T_\sigma$ is 1-nuclear \cite[Theorem 2.3]{delgado2013p}.\label{Base theorem iii}
    \end{enumerate}
\end{proposition}

\indent We can see that the asymptotic behavior of the function $C(\cdot)$ implies some spectral properties of the operator, meanwhile $\omega(\cdot)$ implies the domain where these properties are valid. \\
If we consider the symbol classes $S^\mu_{\rho,\delta}(\hbar\mathbb{Z}^n\times\mathbb{T}^n)$ with $\rho,\delta\geq 0$ of functions $\sigma$ that satisfies
\[\sigma(k,\cdot)\in C^\infty(\mathbb{T}^n) \ \text{for all } k\in\hbar\mathbb{Z},\]
and for all multi-indices $\alpha,\beta$ there exists a positive constant $C_{\alpha,\beta}$ such that
\begin{equation}\label{Symbol class def}
|D_\theta^{(\beta)}\Delta_k^\alpha \sigma(k,\theta)|\leq C_{\alpha,\beta} (1+|k|)^{\mu -\rho\alpha+\delta\beta }, 
\end{equation}
for all $k\in\hbar\mathbb{Z}^n$ and $\theta\in\mathbb{T}^n$. Then it is possible to have an inequality as in the hypothesis of the Proposition \ref{Base theorem} 
\[|(\mathcal{F}_{\mathbb{T}^n}\sigma)(k,m)|\leq C_Q (1+|k|)^{\mu+2\tilde{Q}\delta} \left(1+\frac{1}{\hbar}|m|\right)^{-2\tilde{Q}}, \]
for all $k,m\in\hbar\mathbb{Z}^n$ and $\tilde{Q}\geq 0$. This inequality and Proposition \ref{Base theorem} allow us to obtain sufficient conditions on the order of the symbol to guarantee the boundedness, compactness and $r-$nuclearity of the pseudo-differential operator.\\
\indent The paper is organized as follows. In Section \ref{section matrix analysis} we use some results from infinite matrix theory to conclude simple order conditions to ensure $(p,q)-$boundedness. These conditions are valid for every $0\leq \rho,\delta $ without further restrictions.\\
\indent In Section \ref{section nuclear} we generalize the results on \cite{delgado2015r} to ensure the $r-$nuclearity with the order of the symbol. Our proofs are not based on pseudo-differential calculus allowing to covert exotic cases like $\rho=\delta=0$ or $\rho=\delta=1$.\\
\indent For some classes of symbols we can prove that their associated matrix is almost diagonal, in Section \ref{section approximation} we exploit this property to obtain error estimates of diagonal approximations.\\
\indent Finally, in Section \ref{section applications} we obtain growth estimates for the eigenvalues of the discrete Schrödinger operator and for a general class of elliptic operators. We find that the eigenvalues for discrete anharmonic operators growth faster than in the real case.

\section{Matrix analysis}\label{section matrix analysis}
\indent In this section we will review some of the matrix analysis presented in \cite{maddox1988elements,maddox2006infinite}, specially the boundedness criteria in terms of the elements of the matrix associated to the operator \cite{maddox1988elements}. Every linear operator in $\mathcal{B}(\ell^p)$ with $1\leq p< \infty$ has matrix representation.\\
\indent For the construction of the matrix associated to the operator $T_\sigma$, consider the standard basis $\{e_k\}_{k\in\hbar\mathbb{Z}^n}\subset \ell^p(\hbar\mathbb{Z}^n)$ of elements such that $e_k(x)=1$ if $x=k$ and 0 otherwise. Let $a=\sum_{m\in\hbar\mathbb{Z}^n} a_m e_m $, assuming that $T_\sigma\in\mathcal{B}(\ell^p(\hbar\mathbb{Z}^n))$ we have 
\begin{equation}\label{Matrix rep}
(T_\sigma a)(k)=\sum_{m\in\hbar\mathbb{Z}^n} a_m (T_\sigma e_m)^{(k)},
\end{equation}
where $(T_\sigma e_m)^{(k)}$ is the k-th component of $T_\sigma e_m$, then the associated matrix is $A(k,m)=(T_\sigma e_m)^{(k)}$. The representation (\ref{Matrix rep}) is often called \textit{kernel representation}.
\begin{remark}
    Even if we do not know if the operator is bounded we have the matrix (\ref{Matrix rep}), the existence of the matrix representation is a necessary condition for the operator to be bounded on any $\mathcal{B}(\ell^p)$, but is not sufficient. However, if we find conditions that ensure the boundedness of the operator induced by the matrix, then the operator $T_\sigma$ will be bounded on the dense set $\{e_k\}_{k\in\hbar\mathbb{Z}^n}$,   which implies that $T_\sigma$ is bounded. That is, we can study boundedness conditions for the matrix operator  and these will work on $T_\sigma$ as well.
\end{remark}  

\begin{proposition}
    Every bounded operator $T:\ell^{2}(\hbar\mathbb{Z}^n)\to\ell^{2}(\hbar\mathbb{Z}^n)$ is a pseudo-differential operator with symbol
    \[\sigma(k,\theta)= \sum_{m\in\hbar\mathbb{Z}^n} K(k,m)e^{-2\pi\frac{i}{\hbar}m\theta}, \]
    where $K(\cdot,\cdot)$ is the associated matrix of $T$. 
\end{proposition}
\begin{proof}
    As we mentioned earlier, we can define an associated matrix $K(k,m)=(T_\sigma e_m)^{(k)}$. The columns of the matrix belong to $\ell^{2}(\hbar\mathbb{Z}^n)$, in other words, if we fix $k$ we obtain the function $K(k,\cdot)\in \ell^{2}(\hbar\mathbb{Z}^n) $ for every $k\in\hbar\mathbb{Z}^n$. We define  
     \[\sigma(k,\theta)= \sum_{m\in\hbar\mathbb{Z}^n} K(k,m)e^{-2\pi\frac{i}{\hbar}m\theta},\]
     which converges by the Plancherel formula. Finally, is easy to see that $T_\sigma=T$.
\end{proof}

The following proposition presents boundedness conditions for matrix operators based in the summability of its columns and rows.

\begin{proposition}\label{general bounded criterion}
    Let $A$ be the operator defined by the action of the matrix $(a_{k,m})_{k,m\in\hbar\mathbb{Z}^n}$, we will omit writing the set $\hbar\mathbb{Z}^n$. Let $\frac{1}{p}+\frac{1}{q}=1$, then the following conditions holds
    \begin{enumerate}[(i)]
        \item $A\in\mathcal{B}(\ell^1,\ell^p)$ with $p\in[1,\infty)$ if and only if $\sup_{m}\sum_{k}|a_{k,m}|^p <\infty$ \cite[Theorem 5]{maddox1988elements}.\label{Matrix theorem i}
        \item $A\in\mathcal{B}(\ell^1,\ell^\infty)$ if and only if $\sup_{k,m} |a_{k,m}| <\infty$ \cite[Theorem 5]{maddox1988elements}.\label{Matrix theorem ii}
        \item If $\sum_k\left(\sum_m |a_{k,m}|^q\right)^{\frac{p}{q}}<\infty$ with $1<p<\infty$, then $A\in\mathcal{B}(\ell^p)$ \cite{borwein1979matrix}.\label{Matrix theorem iv} 
    \end{enumerate}
\end{proposition}

 If we return to the Proposition \ref{Base theorem} that motivates the present paper, our first theorem is a consequence of the previous matrix conditions and the existence of an estimative in the Fourier transform of $\sigma$ as
 \begin{equation}\label{fourier inequality}
     |(\mathcal{F}_{\mathbb{T}} \sigma)(k,m)|\leq C(k)\omega(m),
 \end{equation}
for every $k,m\in\hbar\mathbb{Z}^n$.

\begin{theorem}\label{Theorem bounded}
    Let $\sigma$ be a symbol satisfying (\ref{fourier inequality}) and p,q being conjugates. Then the following implications holds:
    \begin{enumerate}[(i)]
        \item $T_\sigma\in\mathcal{B}(\ell^1,\ell^p)$ with $1\leq p<\infty$ if $C(\cdot)\in\ell^p(\hbar\mathbb{Z}^n)$ and $\omega(\cdot)\in\ell^\infty(\hbar\mathbb{Z}^n)$.\label{bounded i}
        \item $T_\sigma\in\mathcal{B}(\ell^1,\ell^\infty)$ if $C(\cdot),\omega(\cdot)\in\ell^\infty(\hbar\mathbb{Z}^n)$.\label{bounded ii}
        \item $T_\sigma\in\mathcal{B}(\ell^p)$ with $1< p$ if $C(\cdot)\in\ell^p(\hbar\mathbb{Z}^n)$ and $\omega(\cdot)\in\ell^q(\hbar\mathbb{Z}^n)$.\label{bounded iv}
     \end{enumerate}
\end{theorem}
\begin{proof}
    Let $a_{k,m}=(\mathcal{F}_{\mathbb{T}^n}\sigma)(k,k-m)$ where the inequality (\ref{fourier inequality}). We will only prove \ref{bounded i}) of the proposition, the other cases are similar. Using the first implication of the previous proposition we obtain
    \begin{align*}
    \sup_{m\in\hbar\mathbb{Z}^n} \sum_{k\in\hbar\mathbb{Z}^n} |a_{k,m}|^p &\leq \sup_{m\in\hbar\mathbb{Z}^n} \|\omega(\cdot)\|_{\ell^\infty(\hbar\mathbb{Z}^n)}^p \sum_{k\in\hbar\mathbb{Z}^n} |C(k)|^p \\
    &=\|C(\cdot) \|_{\ell^p(\hbar\mathbb{Z}^n)} \|\omega(\cdot)\|_{\ell^\infty(\hbar\mathbb{Z}^n)}^p.    
    \end{align*}
\end{proof}

\begin{remark}\label{remark theorem bounded}
      In comparison with previous results, the Theorem \ref{Theorem bounded} presents weaker conditions to ensure boundedness, this is due to the relaxation in the summability of $\omega(\cdot)$. Clearly if $\omega(\cdot)\in\ell^1(\hbar\mathbb{Z}^n)$  the weaker condition for $C(\cdot)$ is the presented in \cite{molahajloo2010pseudo}, where $\omega(\cdot)\in\ell^\infty(\hbar\mathbb{Z}^n)$ implies $p-$boundedness.\\

\end{remark}
\begin{remark}
    In \cite{carlos2011p} is proved that if $a_{k,m}\in \ell^q(\hbar\mathbb{Z}^n\times\hbar\mathbb{Z}^n) $ with $a_{k,m}$ defined as in the previous proof, then  $T_\sigma\in\mathcal{B}(\ell^p)$ with $2< p<\infty$. The Theorem \ref{Theorem bounded}.\ref{bounded iv} is an improvement because $p>q$ when $p>2$, this means that $C(\cdot)\in\ell^p(\hbar\mathbb{Z}^n)$ is a weaker condition that the one $C(\cdot)\in\ell^q(\hbar\mathbb{Z}^n)$ obtained if we had used the result in \cite[Corollary 3.1]{carlos2011p}.
\end{remark}
The existence of an inequality like (\ref{fourier inequality}) may seem a bit arbitrary, but in the context of symbol classes $S^\mu_{\rho,\delta}(\hbar\mathbb{Z}^n\times\mathbb{T}^n)$ this is natural. Moreover, we do not need either infinite derivatives or differences to have an equivalent inequality, just a finite number of derivatives on the toroidal variable as it is show in the following  lemma.
\begin{lemma}\label{fourier bound}
    Let us assume that $\sigma$ satisfies the symbol inequality 
\begin{equation*}
|D_\theta^{(\beta)}\sigma(k,\theta)|\leq C_{\beta} (1+|k|)^{\mu+\delta\beta }, 
\end{equation*}
for $|\beta|\leq Q\in\mathbb{N}$, then 
\[|(\mathcal{F}_{\mathbb{T}^n}\sigma)(k,m)|\leq C_Q (1+|k|)^{\mu+2\tilde{Q}\delta} \left(1+\frac{1}{\hbar}|m|\right)^{-2\tilde{Q}}. \]
for all $k,m\in\hbar\mathbb{Z}^n$ and $\tilde{Q}\leq Q$.
\end{lemma}
\begin{proof}
    The case $(\mathcal{F}_{\mathbb{T}^n}\sigma)(k,0)$ is clear. Let us assume that $m\neq 0$. Then, we have
    \[e^{2\pi \frac{i}{\hbar} m\cdot \theta}= \frac{(1+\mathcal{L}_\theta)}{1+\frac{4\pi^2}{\hbar}|m|^2}e^{2\pi \frac{i}{\hbar} m\cdot \theta},  \]
    where $\mathcal{L}_\theta=\sum^{n}_{j=1}\frac{\partial^2}{\partial \theta_j^2}$ \cite{botchway2020difference}. Using the periodicity of the exponential function and the previous equation we have
    \begin{align*}
        (\mathcal{F}_{\mathbb{T}^n}\sigma)(k,m)&=\int_{\mathbb{T}^n} \sigma(k,\theta)e^{-2\pi\frac{i}{\hbar}m\cdot \theta}d\theta,\\
        &= \int_{\mathbb{T}^n}  \frac{(1+\mathcal{L}_\theta)^{\tilde{Q}} }{(1+\frac{4\pi^2}{\hbar}|m|^2)^{\tilde{Q}}}   \\
        &= \left(1+\frac{4\pi^2}{\hbar}|m|^2\right)^{-\tilde{Q}} \int_{\mathbb{T}^n} e^{2\pi \frac{i}{\hbar} m\cdot \theta} (1+\mathcal{L}_\theta)^{\tilde{Q}}\sigma(k,\theta)d\theta, 
    \end{align*}
    for all $\tilde{Q}\leq Q$. Finally, we have
    \[ |(\mathcal{F}_{\mathbb{T}^n}\sigma)(k,m)|\leq C_Q (1+|k|)^{\mu+2\tilde{Q}\delta} \left(1+\frac{1}{\hbar}|m|\right)^{-2\tilde{Q}}  \]
\end{proof}
In the case of $\sigma\in S^\mu_{\rho,\delta}(\hbar\mathbb{Z}^n\times\mathbb{T}^n)$ the Lemma \ref{fourier bound} is valid for all $\tilde{Q}\in\mathbb{N}$. Now, if $\sigma$ satisfies the inequality of the Lemma \ref{fourier inequality} or if $\sigma\in S^\mu_{\rho,\delta}(\hbar\mathbb{Z}^n\times\mathbb{T}^n)$ we have 
\begin{equation}\label{C()}
    C(k)= C_Q\left( 1+|k| \right)^{\mu+2\tilde{Q}\delta},
\end{equation} 
\begin{equation}\label{w()}
    \omega(m)= \left( 1+\frac{1}{\hbar}|m| \right)^{-2\tilde{Q}},
\end{equation} for every $Q\in\mathbb{N}_0$. This fact allows us to translate the summability conditions in $C(\cdot)$ and $\omega(\cdot)$ into simpler conditions on $\mu,\rho$ and $\delta$.

\begin{corollary}\label{Theorem bounded symbols}
    Assume that $\sigma$ satisfy the following symbol inequality:
    \[|D_\theta^{(\beta)}\sigma(k,\theta)|\leq C_{\beta} (1+|k|)^{\mu+\delta\beta }, \]
    for all $|\beta|\leq \lfloor \frac{n}{2} \rfloor +1$. Let
     p,q be conjugates, then the following implications hold:
    \begin{enumerate}[(i)]
        \item $T_\sigma\in\mathcal{B}(\ell^1,\ell^p)$ with $1\leq p<\infty$ if $\mu<-\frac{n}{p}$.\label{bounded symbol i}
        \item $T_\sigma\in\mathcal{B}(\ell^1,\ell^\infty)$ if $\mu\leq 0$, and this condition is sharp.\label{bounded symbol ii}
        \item $T_\sigma\in\mathcal{B}(\ell^p)$  for every $1\leq p$ if $\mu\leq-(n+2)\delta$, and this condition is sharp for $\delta=0$ \label{bounded symbol iii} 
     \end{enumerate}

\end{corollary}

\begin{proof}
    We only need to use the conditions presented in Theorem \ref{Theorem bounded symbols} and (\ref{Base theorem i}) of Proposition \ref{Base theorem}, along with the fact that $|\cdot|^{-r}\in\ell^1(\hbar\mathbb{Z}^n)$ if and only if $r>n$. In cases (\ref{bounded symbol i}-\ref{bounded symbol ii}) is necessary $\omega(\cdot)\in\ell^\infty(\hbar\mathbb{Z}^n)$, this always holds as a consequence of the definition (\ref{w()}). Meanwhile, to prove (\ref{bounded symbol i}) we see that $C(\cdot)\in\ell^p(\hbar\mathbb{Z}^n)$ if and only if $p(\mu+2Q\delta)<-n$, this inequality is equivalent to 
    \[\mu<-\frac{n}{p}-2Q\delta, \]
    where $Q\in\mathbb{N}_0$ and the lowest value of the right side is when $Q=0$. For the case (\ref{bounded symbol ii}) we only need to $C(\cdot)$ to be bounded, this happens when $\mu+2Q\delta\leq 0$ and again our best choice is $Q=0$.\\
    \indent In order to prove (\ref{bounded symbol iii}) we will need a different strategy. Due the Remark \ref{remark theorem bounded} for (\ref{bounded symbol iii})  we will use Proposition \ref{Base theorem}. It is necessary that 
    \[\mu+2Q\delta\leq 0 \ , \ -2Q<-n.\]
    Thus
    \[\frac{n}{2}<Q\leq -\frac{\mu}{2\delta},\]
    where the condition $Q\in\mathbb{N}_0$ plays a crucial role because we only need 
    \[\left(\frac{n}{2},-\frac{\mu}{2\delta}\right]\cap \mathbb{Z}\neq \emptyset.\]
    This fact is equivalent to the inequality
    \[-\frac{\mu}{2\delta}-\frac{n}{2}\geq 1,\]
    which results on the condition $\mu\leq -(n+2)\delta$. Finally, to prove the sharpness of \ref{bounded symbol ii}-\ref{bounded symbol iii}, we consider the operator 
    $(T_\epsilon f)(k)=|k|^\epsilon f(k),$
    for every $\epsilon>0 $ with symbol $\sigma(k,\theta)=|k|^\epsilon$. We have $\|T_\epsilon e_m \|_{\ell^p(\hbar\mathbb{Z}^n)}=|k|^\epsilon$, then $T_\epsilon$ which has order $\epsilon$ is not bounded in any $\ell^p(\hbar\mathbb{Z}^n)$. 

\end{proof}
\begin{corollary}\label{lp bound special case}
    If $\sigma$ satisfies the hypothesis of the previous theorem, $\delta=0$ and $\mu\leq 0$. Then $T_\sigma$ is $(p,p)$ bounded for every $p\geq 1$.
\end{corollary}
The Corollary \ref{lp bound special case} coincides with the equivalent results presented in \cite[Theorem 5.2]{botchway2020difference} and \cite[Theorem 2.8]{rabinovich2004pseudodifference}. Moreover, since our proof is not based on the pseudo-differential calculus, the conditions presented in Corollary \ref{Theorem bounded symbols} include the exotic cases such as $\rho=\delta=0$ or $\rho=\delta=1$. 

\indent Finally, we will obtain a similar result with the compactness criterion given in Proposition \ref{Base theorem} .

\begin{corollary}\label{compact criterion}
    If $\sigma$ satisfies the hypothesis of the Corollary \ref{Theorem bounded symbols} and 
    \[\mu<-(n+2)\delta.\]
    Then $T_\sigma:\ell^p(\hbar\mathbb{Z}^n)\to\ell^p(\hbar\mathbb{Z}^n)$ is a compact operator for every $ p\geq 1$. Additionally, this inequality is sharp for $\delta=0$.
\end{corollary}
\begin{proof}
    We only prove the sharpness because the order condition is equivalent to the previous corollary. Consider the difference operator
    \[(Tf)(k)=f(k+\hbar)-f(k) \text{ with } k\in \hbar\mathbb{Z},\]
     whose symbol $\sigma(k,\theta)=e^{2\pi i \theta} -1 $ belongs to $S^0_{1,0}(\hbar\mathbb{Z}\times\mathbb{T})$. We need to prove that $T$ is not compact in any $\ell^p(\hbar\mathbb{Z})$. If we assume that $T$ is compact and consider the basis $\{e_i\}_{i\in\hbar\mathbb{Z}}\subset \ell^p(\hbar\mathbb{Z})$ which is a bounded sequence in every $\ell^p(\hbar\mathbb{Z})$, but we have that 
    
    \[ (Te_i)(k)=  \begin{cases} 
      1 & k=i-\hbar, \\
      -1 & k=i, \\
      0 & otherwise. 
   \end{cases}
\]
    The sequence $\{Te_i\}_{i\in\hbar\mathbb{Z}}$ does not contain any convergent subsequence, this fact shows that $T$ is not compact in any $\ell^p(\hbar\mathbb{Z})$. Equivalent operators can be obtained for higher dimensions, which makes the inequality sharp for $\delta=0$. 
    
\end{proof}

\section{r-Nuclearity conditions}\label{section nuclear}
The order conditions to ensure the r-nuclearity will help us obtain estimates for the growth of the eigenvalues, even if we only know the order of the pseudo-differential operator. \\
First, we recall the definition of an r-nuclear operator presented in \cite{delgado2015r}. Let $A:L^{p_1}(X_1,\mu_1)\to L^{p_2}(X_2,\mu_2)$ be a bounded operator with $(X_1,\mu_1)$ and $(X_2,\mu_2)$ being $\sigma$-finite measure spaces, $A$ is r-nuclear with $0<r\leq 1$ if it admits a representation 
\begin{equation}
    (Af)(x)=\int_{X_1}\left(\sum_{m=1}^{\infty} g_m(x)h_m(y) \right)f(y)d\mu_1(y), 
\end{equation}
such that 
\begin{equation}
    \sum_{m=1}^{\infty} \|g_m\|^r_{L^{p_2}(X_2,\mu_2)} \|h_m\|^r_{L^{p'_1}(X_1,\mu_1)}<\infty,
\end{equation}
where $1\leq p_1,p_2<\infty$ and $\frac{1}{p_1}+\frac{1}{p'_1}=1$.\\
\indent Following the strategy presented in \cite{delgado2013p} to develop sufficient conditions to ensure that the operator is 1-nuclear, we obtain the following theorem for r-nuclearity. 
\begin{theorem}
    Let $K:\hbar\mathbb{Z}^n\times \hbar\mathbb{Z}^n\to \mathbb{C}$ be a function such that 
    \[\sum_{k\in\hbar\mathbb{Z}^n}\left(\sum_{m\in\hbar\mathbb{Z}^n} |K(k,m)|^{p_2} \right)^{r/p_2} <\infty. \]
    Let $A:\ell^{p_1}(\hbar\mathbb{Z}^n)\to\ell^{p_2}(\hbar\mathbb{Z}^n)$, $1\leq p_1,p_2<\infty$, be the linear operator defined by 
    \[e_m(Ae_k)=K(k,m), \ \ \ \  k,m\in\hbar\mathbb{Z}^n,\]
    where $e_m$ is considered to be an element in $\ell^{p'_2}(\hbar\mathbb{Z}^n)$. Then A is a nuclear operator.
\end{theorem}
\begin{proof}
    For all $f\in\ell^{p_1}(\hbar\mathbb{Z}^n)$, we have
    \[f=\sum_{k\in\hbar\mathbb{Z}^n} e_k(f)e_k, \]
    then 
    \[Af=\sum_{k\in\hbar\mathbb{Z}^n} e_k(f)Ae_k. \]
    Using Minkowski's inequality to obtain 
    \begin{align*}
        \|Af \|_{\ell^{p_2}(\hbar\mathbb{Z}^n)}&\leq \sum_{k\in\hbar\mathbb{Z}^n} \|e_k(f)Ae_k \|_{\ell^{p_2}(\hbar\mathbb{Z}^n)}\\
        &\leq \sup_{k} \|f(k)\| \sum_{k\in\hbar\mathbb{Z}^n} \|Ae_k \|_{\ell^{p_2}(\hbar\mathbb{Z}^n)}\\
        & \leq \|f\|_{\ell^{\infty}(\hbar\mathbb{Z}^n)}  \sum_{k\in\hbar\mathbb{Z}^n}\left(\sum_{m\in\hbar\mathbb{Z}^n} |K(k,m)|^{p_2} \right)^{1/p_2} <\infty,
    \end{align*}
    and ensure the boundedness of  $A$. Moreover, the decomposition of $A$ satisfies 
    \begin{align*}
    \sum_{k\in\hbar\mathbb{Z}^n} \|Ae_k\|^r_{\ell^{p_2}(\hbar\mathbb{Z}^n)} \|e_k\|^r_{\ell^{p'_1}(\hbar\mathbb{Z}^n)}&=\sum_{k\in\hbar\mathbb{Z}^n} \|Ae_k\|^r_{\ell^{p_2}(\hbar\mathbb{Z}^n)}\\
    &\leq \sum_{k\in\hbar\mathbb{Z}^n}\left(\sum_{m\in\hbar\mathbb{Z}^n} |K(k,m)|^{p_2} \right)^{r/p_2}<\infty
    \end{align*}
\end{proof}
\begin{corollary}
    Let $A:\ell^{p_1}(\hbar\mathbb{Z}^n)\to\ell^{p_2}(\hbar\mathbb{Z}^n)$ be defined as in the previous theorem. If there exists functions $C(\cdot)\in \ell^{r}(\hbar\mathbb{Z}^n)$ and $\omega(\cdot)\in \ell^{p_2}(\hbar\mathbb{Z}^n)$ such that
    \[|K(k,m)|\leq C(k)\omega(m), \ \ \  k,m\in\hbar\mathbb{Z}^n.\]
    Then $A$ is r-nuclear with $0<r\leq 1.$
\end{corollary}
The following corollary is a direct consequence of the previous one and the decay rate obtained in \cite{reinov2013distribution} for the eigenvalues of r-nuclear operators.
\begin{corollary}\label{eigenvalue decay}
    Let $\sigma$ be a symbol that satisfies the hypothesis of Corollary \ref{Theorem bounded symbols}. If 
    \[\mu<-\frac{n}{r}-\left(\frac{n}{p_2}+2  \right)\delta. \]
     Then $T_\sigma:\ell^{p_1}(\hbar\mathbb{Z}^n)\to\ell^{p_2}(\hbar\mathbb{Z}^n)$ is r-nuclear with $0<r\leq 1$ and $1\leq p_1,p_2<\infty$. Furthermore we have the following rate of decay for the eigenvalues when $p_1=p_2=p$
    \[ \lambda_j(T_\sigma)=O\left(j^{-\frac{1}{t}}\right), \]
    where 
    \[\frac{1}{r}=\frac{1}{t}+\left|\frac{1}{p}-\frac{1}{2}\right|.\]
\end{corollary}

\section{First order approximation}\label{section approximation}
In Corollary \ref{compact criterion} we have already seen that for $\mu< -(n+2)\delta$ we have 
\[  \left| (\mathcal{F}_{\mathbb{T}} \sigma)(k,k-m) \right| \leq C_{\lfloor\frac{n}{2}\rfloor +1}  ( 1+|k| )^{\mu+(n+2)\delta}\left( 1+\frac{1}{\hbar}|k-m| \right)^{-n-2},\]
for all $k,m\in \hbar\mathbb{Z}^n $. This bound indicates that the associated matrix is \textit{almost diagonal}, in other words, we can write $T_\sigma:\ell^{2}(\hbar\mathbb{Z}^n)\to\ell^{2}(\hbar\mathbb{Z}^n)$ as 
\[T_\sigma=D_\sigma+R,\]
where $D_\sigma$ is a diagonal matrix with entries 
$(D_\sigma)_{k,k}=\int_{\mathbb{T}^n} \sigma(k,\theta)d\theta,$
and R is a residue with $\|R\|<\epsilon$, both are bounded as consequence of the Proposition \ref{general bounded criterion} .\\
This fact allows us to treat $T_\sigma$ as a perturbation of $D_\sigma$. Now let $\{e_k(\cdot)\}_{k\in\hbar\mathbb{Z}^n}$ be the standard basis of every $\ell^p(\hbar\mathbb{Z}^n)$, those are the eigenvectors of $D_\sigma$
\[D_\sigma e_k=\lambda_k e_k,\]
\[\lambda_k=\int_{\mathbb{T}^n} \sigma(k,\theta)d\theta. \]
We can write the eigenvalues and eigenvectors $(\widetilde{\lambda}_k,\phi_k)$ of $T_\sigma$ as
\[\widetilde{\lambda}_k=\lambda_k +\Delta\lambda_k,\]
\[\phi_k=e_k +\Delta e_k,\]
where $\{\phi_k\}_{k\in\hbar\mathbb{Z}^n}$ is an orthonormal basis. Hence
\[(D_\sigma+R )(e_k+\Delta e_k)=(\lambda_k+\Delta\lambda_k)(e_k+\Delta e_k)\]
\begin{equation}\label{EigenAprox}D_\sigma(\Delta e_k)+R(e_k)=\lambda_k\Delta 
e_k+\Delta\lambda_k e_k+E,
\end{equation}
with $E=(\Delta\lambda_kI-R)(\Delta e_k)$. Using the fact that $\{e_k(\cdot)\}_{k\in\hbar\mathbb{Z}^n}$ is an orthonormal basis we can write the vectors $\Delta\phi_k$ as
\[ \Delta e_k=\sum_{j\in\hbar\mathbb{Z}^n} \epsilon_{k,j} e_j.  \]

\noindent This fact together with equation (\ref{EigenAprox}) and the continuity of $D_\sigma$ allows us to obtain 
\[\left(\sum_{j\in\hbar\mathbb{Z}^n} \epsilon_{k,j}\lambda_j e_j \right)+R(e_k)=  \left(\sum_{j\in\hbar\mathbb{Z}^n} \epsilon_{k,j}\lambda_k e_j \right)+\Delta\lambda_k e_k+E.\]
If we multiply it by $e_k$ we conclude that 
\[\braket{R(e_k),e_k}=\Delta\lambda_k+\braket{E,e_k}.\]
As we define $R$ to be the off-diagonal of $T_\sigma$ is clear that $\braket{R(e_k),e_k}=0$, moreover if we assume $\epsilon_{k,k}\neq -1$, then
\[ \Delta\lambda_k= -\braket{E,e_k}=\left(\sum_{j\in\hbar\mathbb{Z}^n} \epsilon_{k,j} \braket{Re_j,e_k}\right)-\epsilon_{k,k}\Delta\lambda_k \]
\[\Delta\lambda_k=\frac{1}{1+\epsilon_{k,k}}\sum_{j\in\hbar\mathbb{Z}^n} \epsilon_{k,j} \braket{Re_j,e_k}. \]
Since $|\epsilon_{k,j}|\leq 1$ for $j\neq k$ and the convergence of $\sum_{j\in\hbar\mathbb{Z}^n} |\epsilon_{k,j}|^2  $, we can ensure for a sufficiently large $J(k)>0$ we have 
\[\frac{|\epsilon_{k,j}|}{|1+\epsilon_{k,k}|}<1, \ \ j>J(k). \]
Finally, for some constant $C_k>0$ we have 
\[|\Delta\lambda_k|\leq C_k\sum_{j\in\hbar\mathbb{Z}^n}  |\braket{Re_j,e_k}|=C_k\sum_{j\in\hbar\mathbb{Z}^n}  | (\mathcal{F}_{\mathbb{T}} \sigma)(k,k-j)|, \]
\[|\Delta\lambda_k|\leq C_{\lfloor\frac{n}{2}\rfloor +1,k} \ ( 1+|k| )^{\mu+(n+2)\delta}\sum_{j\in\hbar\mathbb{Z}^n} \left( 1+\frac{1}{\hbar}|j| \right)^{-n-2}.\]
But the constant depends strictly on $k$  in the case that $\epsilon_{k,k}\backsim -1$. To ensure that $\epsilon_{k,k}$ is not close to $-1$ we need to use the following equation
\begin{equation}\label{Error equation}
2R(\epsilon_{k,k})+\|\Delta e_k\|^2=0,
\end{equation}
which is a consequence of $\|\phi_k\|=\|e_k+\Delta e_k\|=1$. Then we only need to assume that $\|\Delta e_k\|\leq 1 $. If we do not assume anything on the norm $\|\Delta e_k\|$ we could have that $\|\Delta e_k\|\backsim 2$, which by the equation (\ref{Error equation}) is equivalent to $\epsilon_{k,k}\backsim -1$, making impossible to bound $\Delta \lambda_k$. 

\begin{lemma} If $\mu<-(n+2)\delta$ and assuming that $\|\Delta e_k\|\leq 1 $, then the eigenvalues $\lambda_k$ of $T_\sigma:\ell^{2}(\hbar\mathbb{Z}^n)\to\ell^{2}(\hbar\mathbb{Z}^n)$ can be approximated by 
\[ \lambda_k=\int_{\mathbb{T}^n} \sigma(k,\theta)d\theta + O(|k|^{\mu+(n+2)\delta}).\]
\end{lemma}
From other perspective, when we see the eigenvector $\phi_k$ as an eigenvector approximation of $D_\sigma$ we obtain the following lemma, which together with $D_\sigma$ being compact allow us to conclude that for large $k$ the eigenvectors of $T_\sigma$ are approximations of eigenvectors of $D_\sigma$.
\begin{lemma} Under the same assumptions as in the previous lemma, then 

\begin{equation} 
\min_{\widetilde{k}\in\hbar\mathbb{Z}^n} \left|\widetilde{\lambda}_k-\int_{\mathbb{T}^n} \sigma(\widetilde{k},\theta) d\theta  \right| \leq  \|D_\sigma\phi_k - \widetilde{\lambda}_k\phi_k\|\leq \min \left\{\max_{\widetilde{k}\in\hbar\mathbb{Z}^n} \left|\widetilde{\lambda}_k-\int_{\mathbb{T}^n} \sigma(\widetilde{k},\theta) d\theta  \right|,\|R\|\right\},
\end{equation}
\end{lemma}
\begin{proof}
Due to $T_\sigma\phi_k=\widetilde{\lambda}_k\phi$, we have 
    \[(D_\sigma - \widetilde{\lambda}_kI)\phi_k=-R\phi_k, \]
where we can use that $(D_\sigma - \widetilde{\lambda}_kI)$ is diagonal, then it's norm can be bounded by the minimum and maximum values of the diagonal. 
\[\min_{\widetilde{k}\in\hbar\mathbb{Z}^n} \left|\widetilde{\lambda}_k-\int_{\mathbb{T}^n} \sigma(\widetilde{k},\theta) d\theta  \right| \leq  \|(D_\sigma - \widetilde{\lambda}_kI)\phi_k\|\leq \max_{\widetilde{k}\in\hbar\mathbb{Z}^n} \left|\widetilde{\lambda}_k-\int_{\mathbb{T}^n} \sigma(\widetilde{k},\theta) d\theta  \right|.\]
Finally, we have that $\|(D_\sigma - \widetilde{\lambda}_kI)\phi_k\|=\|R\|$ due to the first identity and the boundedness of $R$. 
\end{proof}

\section{Discrete Schrödinger operators}\label{section applications}
Let $H_{\hbar,V}$ be the discretization of the  Schrödinger operator in $\mathbb{R}^n$, defined as
\begin{equation}\label{Schrodinger operator}
(H_{\hbar,V}f)(k)=-\hbar^{-2}\Delta f+ V(k)f(k) \text{  for } k\in\hbar\mathbb{Z}^n \text{ and } V(\cdot)\geq0,
\end{equation}
where $\Delta$ is the discrete Laplacian defined as 
\[ (\Delta f)(k)=\sum_{j=1}^{n}f(k+\hbar e_j)-2f(k)+f(k-\hbar e_j),   \]
and $V(k)\geq 0$ being polynomial with order $\mu>0$.\\
\indent The operator $H_{\hbar,V}: \text{Dom}\ (H_{\hbar,V})\to \ell^2(\hbar\mathbb{Z}^n)$ is densely defined and self-adjoint with the domain $\text{Dom}\ (H_{\hbar,V})$ defined as
\[\text{Dom}\ (H_{\hbar,V})=\left\{u\in\ell^2(\hbar\mathbb{Z}^n): (H_{\hbar,V}+I)u\in \ell^2(\hbar\mathbb{Z}^n) \right\}. \]
\noindent In order to study the eigenvalues of the operator (\ref{Schrodinger operator}), we consider the perturbed version $(H_{\hbar,V}+\lambda I)\in \Psi^\mu_{1,0}(\hbar\mathbb{Z}^n\times\mathbb{T}^n) $ for some $\lambda>0$. This choice is due to the previous results obtained in \cite{dasgupta2023discrete} that guarantee its invertibility.

\begin{proposition}[\cite{dasgupta2023discrete}]
    Let $V\geq 0$ and $|V(k)|\to \infty$ as $|k|\to\infty$. Then $H_{\hbar,V}$ has a purely discrete spectrum. 
\end{proposition}

This ensures that we can choose some $\lambda<0$ in the resolvent set such that $(H_{\hbar,V}+\lambda I)^{-1}$ is bounded, where the symbol of $(H_{\hbar,V}+\lambda I)$ is
\begin{equation}\label{symbol schrodinger}
\sigma(k,\theta)=-2\hbar^{-2}\cos{2\pi\theta} +V(k)+(\lambda+2).
\end{equation}
We recall the definition of elliptic symbols. A symbol $\sigma\in S^{\mu}_{\rho,\delta}(\hbar\mathbb{Z}^n\times\mathbb{T}^n) $ is called elliptic (of order $\mu$) if there exist $C>0$ and $M>0$ such that 
\[|\sigma(k,\theta)|\geq C(1+|k|)^\mu\]
for all $\theta\in\mathbb{T}^n$ and for $|k|\geq M$, $k\in\hbar\mathbb{Z}^n$ \cite[Definition 4.5]{botchway2024semi}.
Clearly \eqref{symbol schrodinger} is elliptic with order $\mu$. Then, there exists a parametrix $P\in\Psi^{-\mu}_{1,0}(\hbar\mathbb{Z}^n\times\mathbb{T}^n)$ \cite[Theorem 4.7]{botchway2024semi}. These facts allows us to conclude our main result. 
\begin{theorem}
    There exists some $\lambda>0$ such that $(H_{\hbar,V}+\lambda I)^{-1}$ with $\mu\geq 1$ is r-nuclear. Moreover, we have the following rate of growth for the eigenvalues of $H_{\hbar\mathbb{Z}}$:
    \[C j^{\frac{1}{r}} \leq \lambda_j(H_{\hbar\mathbb{Z}})  \text{ as } j\to\infty, \]
    with $\frac{1}{\mu}<r\leq 1$.
\end{theorem}
\begin{proof}
    The operator$(H_{\hbar,V}+\lambda I)$  is elliptic, and by the calculus of the class $\Psi^{-\mu}_{1,0}(\hbar\mathbb{Z}\times\mathbb{T})$ exposed in \cite[Theorem 4.7]{botchway2024semi} there exists an operator $P\in \Psi^{-\mu}_{1,0}(\hbar\mathbb{Z}^n\times\mathbb{T}^n)$ such that 
    \[ (H_{\hbar,V}+\lambda I)P\backsim P(H_{\hbar,V}+\lambda I)\backsim I \textbf{ modulo } \Psi^{\infty}_{1,0}(\hbar\mathbb{Z}\times\mathbb{T}).  \]
    In other words, we have 
    \[ P(H_{\hbar,V}+\lambda I)=I -R,  \]
    where $R\in\Psi^{\infty}_{1,0}(\hbar\mathbb{Z}\times\mathbb{T})$. Due to the existence of the inverse of $(H_{\hbar,V}+\lambda I) $ we have
    \[P=(H_{\hbar,V}+\lambda I)^{-1}+R(H_{\hbar,V}+\lambda I)^{-1}, \]
    \begin{equation}\label{Parametrix-inverse relation}
    (H_{\hbar,V}+\lambda I)^{-1}=P-R(H_{\hbar,V}+\lambda I)^{-1}.
    \end{equation}
    Now we can use the order condition given in Corollary \ref{eigenvalue decay} to ensure that $P$ and $R$ are $r-$nuclear operators. \\
    \indent We are dealing with operators defined on $\ell^2(\hbar\mathbb{Z})$, where the definition of $r-$nuclearity coincides with the definition for $r-$Schatten-von Neumann operators \cite{oloff}, and these operators form a two-sided ideal of $\mathcal{B}(\ell^2(\hbar\mathbb{Z}^n))$. We have $(H_{\hbar,V}+\lambda I)^{-1}$ bounded, so we conclude by the ideal properties that $R(H_{\hbar,V}+\lambda I)^{-1}$ is $r-$nuclear for $\frac{1}{\mu}<r$. Using the equation (\ref{Parametrix-inverse relation}) we have $(H_{\hbar,V}+\lambda I)^{-1}$ is $r-$nuclear because is the sum of $r-$nuclear operators.\\
    Finally, we have for some $C>0$ 
    \[ \lambda_j\left( (H_{\hbar\mathbb{Z}}+\lambda I)^{-1}\right) \leq C j^{-\frac{1}{r}}    \text{ as } j\to\infty, \]
    which is equivalent to 
    \[  C j^{\frac{1}{r}}\leq \lambda_j\left( H_{\hbar\mathbb{Z}}+\lambda I\right)    \text{ as } j\to\infty, \]
    \[  \widetilde{C} j^{\frac{1}{r}}\leq \lambda_j\left( H_{\hbar\mathbb{Z}}\right)    \text{ as } j\to\infty. \]
    
\end{proof}
\noindent Compared with the results of  \cite{chatzakou2021class} for the real case $V(x)=|x|^{2l}$, where the growth estimate work for $r>\frac{k+1}{2k}$. We see that in the discrete case the growth estimate of the eigenvalues is larger because $r>\frac{1}{2k}$.\\
\indent Finally, an equivalent proof can be made to obtain a more general theorem for elliptic operators with positive order. 
\begin{theorem}
    Let $E\in \Psi^{\mu}_{\rho,\delta}(\hbar\mathbb{Z}^n\times\mathbb{T}^n)$ be elliptic with $\mu\geq 1 $ and $0\leq \delta<\rho\leq 1$. If $\rho(E)\cap(-\infty,0)\neq \emptyset$, then $(E-\lambda I)^{-1}$ with $\lambda\in\rho(E)\cap(-\infty,0)$ is $r-$nuclear and there exists a constant $C>0$ such that
     \[C j^{\frac{1}{r}} \leq \lambda_j(E)  \text{ as } j\to\infty, \]
    with $\frac{1}{\mu}<r\leq 1$.
\end{theorem}

\backmatter

------------------------------------

\section*{Declarations}

\begin{itemize}
\item No funding
\item No conflict of interest/Competing interests (check journal-specific guidelines for which heading to use)
\end{itemize}

\noindent


%%====================================================
\bibliography{sn-article}% common bib file

\end{document}